\documentclass[11pt,reqno]{amsart}
\usepackage{color}

\usepackage{amsmath}
\usepackage{amsfonts,amscd}
\usepackage{amssymb}
\usepackage{url}
\usepackage{hyperref}

\usepackage[english]{babel}
\usepackage{booktabs}
\usepackage{tikz}

\theoremstyle{plain}
\newtheorem{theorem}                 {Theorem}      [section]

\newtheorem{lemma}        [theorem]  {Lemma}

\theoremstyle{definition}

\newtheorem{definition}   [theorem]  {Definition}

\newtheorem{example}      [theorem]  {Example}

\newtheorem{notation}     [theorem]  {Notation}

\numberwithin{equation}{section}

\def \cn{{\mathbb C}}

\def \qn{{\mathbb Q}}
\def \rn{{\mathbb R}}

\def \E{\mathcal E}

\def \L{\mathcal L}

\def \P{\mathcal P}

\def\nab#1#2{\hbox{$\nabla$\kern -.3em\lower 1.0 ex
		\hbox{$#1$}\kern -.1 em {$#2$}}}
\def\hatnab#1#2{\hbox{$\nabla$\kern -.3em\lower 1.0 ex
		\hbox{$#1$}\kern -.1 em {$#2$}}}

\def \SL2{\widetilde{\text{\bf SL}}_{2}(\rn)}

\DeclareMathOperator{\Div}{div}

\numberwithin{equation}{section}
\allowdisplaybreaks

\begin{document}
	
	
\title
[New Harmonic Morphisms via rational exponents]
{New Complex-Valued Harmonic Morphisms\\ via Rational Exponents}


\begin{abstract}
The method of complex-valued eigenfamilies has been used to solve difficult non-linear problems in differential geometry.  The purpose of this note is to generalise the construction method from the use of polynomial solutions to those obtained by rational exponents.  This gives many new solutions to the problem under investigation.
\end{abstract}

\author{Sigmundur Gudmundsson}
\address{Mathematics, Faculty of Science\\
	University of Lund\\
	Box 118, Lund 221 00\\
	Sweden}
\email{Sigmundur.Gudmundsson@math.lu.se}

\maketitle	

\section{Introduction}
\label{section-introduction}

The study of minimal submanifolds of a given ambient space plays a central role in differential geometry.  This has a long, interesting history and has attracted the interests of profound mathematicians for many generations.  The famous Weierstrass-Enneper representation formula, for minimal surfaces in three-dimensional Euclidean space, brings {\it complex analysis} into play as a useful tool for the study of these beautiful objects.

This was later generalised to the study of minimal surfaces in much more general ambient manifolds via {\it harmonic conformal immersions}.  The next  result follows from the seminal paper \cite{Eel-Sam} of Eells and Sampson from 1964.  For this see also Proposition 3.5.1 of \cite{Bai-Woo-book}.

\begin{theorem}
	Let $\phi:(M^m,g)\to (N,h)$ be a smooth conformal map between Riemannian manifolds.  If $m=2$ then $\phi$ is harmonic if and only if the image is minimal in $(N,h)$.
\end{theorem}

This result has turned out to be very useful in the construction of minimal surfaces in Riemannian symmetric spaces of various types.  For this we refer to \cite{Cal},
\cite{Eel-Woo} and \cite{Bur-Gue}, just to name a few.
\smallskip

In their work \cite{Bai-Eel} from 1981, Baird and Eells have shown that complex-valued harmonic morphisms from Riemannian manifolds are useful tools for the study of minimal submanifolds of codimension two.

\begin{theorem}\label{theorem-Bai-Eel-special}\cite{Bai-Eel}
Let $\phi:(M,g)\to\cn$ be a complex-valued harmonic morphism from a Riemannian manifold.  Then every regular fibre of $\phi$ is a minimal submanifold of $(M,g)$ of codimension {\it two}.
\end{theorem}

This can be seen as dual to the above-mentioned generalisation of the Weierstrass-Enneper representation. Harmonic morphisms are the much studied {\it horizontally conformal harmonic maps}.  For an introduction to the general theory we recommened the book \cite{Bai-Woo-book}, by Baird and Wood, and the regularly updated online bibliography \cite{Gud-bib}.
\medskip

The method of complex-valued eigenfamilies was introduced by Anna Sakovich and the current author in 2008, see \cite{Gud-Sak-1}.  This has since been used to construct explicit solutions to difficult non-linear problems in differential geometry.  The most notable cases is that of complex-valued harmonic morphisms on Riemannian manifolds.  These constructions are based on the following result.

\begin{theorem}\cite{Gud-Sak-1}\label{theorem-polynomial}
Let $(M,g)$ be a Riemannian manifold and 
$$
\E =\{\phi_1,\dots,\phi_n\}
$$ 
be a finite eigenfamily of complex-valued functions on $M$. If $P,Q:\cn^n\to\cn$ are linearily independent homogeneous polynomials of the same positive degree then the quotient
$$
\frac{P(\phi_1,\dots ,\phi_n)}{Q(\phi_1,\dots ,\phi_n)}
$$ 
is a non-constant harmonic morphism on the open and dense subset 
$$
\Omega (Q)=\{p\in M\,|\, Q(\phi_1(p),\dots ,\phi_n(p))\neq 0\}.
$$
\end{theorem}

It is therefore interesting to construct explicit  eigenfamilies (see Definition \ref{definition-eigenfamily}) of complex-valued functions on Riemannian manifolds.  For such objects on various Lie groups and symmetric spaces,  see for example \cite{Gud-Sak-1}, \cite{Gud-Sak-2} and  \cite{Gha-Gud-5}.

The main aim of this work is to improve the construction method by generalising Theorem \ref{theorem-polynomials} below, to the following useful result.

\begin{theorem}\label{theorem-rational}
	Let $\E=\{\phi_1,\phi_2,\dots,\phi_n\}$ be a $(\lambda,\mu)$-family on a Riemannian manifold $(M,g)$,  $d\in\qn^+$ and $\alpha,\beta\in \L^{n-1}_d$ with 
	$$
	\L^{n-1}_d=\{\alpha=(\alpha_1,\alpha_2,\dots,\alpha_n)\in\qn^n\,|\,\alpha_1+\alpha_2+\cdots +\alpha_n=d,\ \alpha_k\in \qn^+_0\}.
	$$
	Further, let the complex-valued functions $\Phi_\alpha:U_\alpha\to\cn$ and $\Phi_\beta:U_\beta\to\cn$ be defined on the appropriate open subsets $U_\alpha,U_\beta$ of $M$ by
	$$
	\Phi_\alpha=\phi_1^{\alpha_1}\phi_2^{\alpha_2}\cdots\phi_n^{\alpha_n}\ \ \text{and}\ \ \Phi_\beta=\phi_1^{\beta_1}\phi_2^{\beta_2}\cdots\phi_n^{\beta_n},
	$$
	respectively.  Then the conformality operator $\kappa$ and the tension field $\tau$ satisfy
	$$
	\kappa(\Phi_\alpha,\Phi_\beta)=d^2\mu\cdot \Phi_\alpha\, \Phi_\beta\ \ \text{and}\ \ 
	\tau(\Phi_\alpha)=(d\,\lambda+d(d-1)\,\mu)\cdot \Phi_\alpha.
	$$
\end{theorem}

\begin{notation}  Every complex number $z\in\cn\setminus\rn^-_0$ has a unique presentation in polar coordinates as $z=r\cdot e^{i\theta}$, where $r\in\rn^+$ is a positive real number and $\theta\in (-\pi,+\pi)$.  For a rational number $q/p\in\qn$, with $p,q$ relatively prime, we use the notation
$$z^{p/q}=r^{p/q}\cdot e^{i\theta p/q}.$$
\end{notation}  

\section{Eigenfunctions and Eigenfamilies}
\label{section-eigenfunctions}

Let $(M,g)$ be an $m$-dimensional Riemannian manifold and $T^{\cn}M$ be the complexification of the tangent bundle $TM$ of $M$. We extend the metric $g$ to a complex-bilinear form on $T^{\cn}M$.  Then the gradient $\nabla\phi$ of a complex-valued function $\phi:(M,g)\to\cn$ is a section of $T^{\cn}M$.  In this situation, the well-known complex linear {\it Laplace-Beltrami operator} (alt. {\it tension field}) $\tau$ on $(M,g)$ acts locally on $\phi$ as follows
$$
\tau(\phi)=\Div (\nabla \phi)=\sum_{i,j=1}^m\frac{1}{\sqrt{|g|}} \frac{\partial}{\partial x_j}
\left(g^{ij}\, \sqrt{|g|}\, \frac{\partial \phi}{\partial x_i}\right).
$$
For two complex-valued functions $\phi,\psi:(M,g)\to\cn$ we have the following well-known fundamental relation
\begin{equation}\label{equation-basic-1}
\tau(\phi\cdot \psi)=\tau(\phi)\cdot\psi +2\cdot\kappa(\phi,\psi)+\phi\cdot\tau(\psi),
\end{equation}
where the complex bilinear {\it conformality operator} $\kappa$ is given by 
$$
\kappa(\phi,\psi)=g(\nabla \phi,\nabla \psi).
$$ 
Locally this satisfies 
$$\kappa(\phi,\psi)=\sum_{i,j=1}^mg^{ij}\cdot\frac{\partial\phi}{\partial x_i}\frac{\partial \psi}{\partial x_j}.$$
A useful formula for the symmetric conformality operator $\kappa$ is the following
\begin{equation}\label{equation-basic-2}
\kappa(\phi,\psi\cdot\xi)=\kappa(\phi,\psi)\cdot\xi+\psi\cdot\kappa(\phi,\xi).
\end{equation}
\smallskip

\begin{definition}\cite{Gud-Sak-1}\label{definition-eigenfamily}
Let $(M,g)$ be a Riemannian manifold. Then a complex-valued function $\phi:M\to\cn$ is said to be an {\it eigenfunction} if it is eigen both with respect to the Laplace-Beltrami operator $\tau$ and the conformality operator $\kappa$ i.e. there exist complex numbers $\lambda,\mu\in\cn$ such that $$\tau(\phi)=\lambda\cdot\phi\ \ \text{and}\ \ \kappa(\phi,\phi)=\mu\cdot \phi^2.$$	
A set $\E =\{\phi_i:M\to\cn\ |\ i\in I\}$ of complex-valued functions is said to be an {\it eigenfamily} on $M$ if there exist complex numbers $\lambda,\mu\in\cn$ such that for all $\phi,\psi\in\E$ we have 
$$\tau(\phi)=\lambda\cdot\phi\ \ \text{and}\ \ \kappa(\phi,\psi)=\mu\cdot \phi\,\psi.$$ 
\end{definition}

\begin{example}\cite{Gud-Mun-1}
	\label{example-basic-sphere} 
	Let $S^{2n-1}$ be the odd-dimensional unit sphere in the standard Euclidean space $\cn^{n}\cong\rn^{2n}$ and define $\phi_1,\dots,\phi_n:S^{2n-1}\to\cn$ by
	$$\phi_j:(z_1,\dots,z_{n})\mapsto \frac{z_j}{\sqrt{|z_1|^2+\cdots +|z_n|^2}}.$$  Then the tension field $\tau$ and the conformality operator $\kappa$ on $S^{2n-1}$ satisfy	$$\tau(\phi_j)=-\,(2n-1)\cdot\phi_j\ \ \text{and}\ \ \kappa(\phi_j,\phi_k)=-\,1\cdot \phi_j\cdot\phi_k.$$
\end{example}

For the standard complex projective space $\cn P^n$ we similarly have a complex multi-dimensional eigenfamily.

\begin{example}\cite{Gud-Mun-1}\label{example-basic-projective-space}
	Let $\cn P^n$ be the standard $n$-dimensional complex projective space. For a fixed integer $1\le\alpha < n+1$ and some $1\le j\le\alpha < k\le n+1$  define the function $\phi_{jk}:\cn P^n\to\cn$ by
	$$\phi_{jk}:[z_1,\dots,z_{n+1}]\mapsto \frac
	{z_j\cdot\bar z_k}{z_1\cdot \bar z_1+\cdots + z_{n+1}\cdot \bar z_{n+1}}.$$  
	Then the tension field $\tau$ and the conformality operator $\kappa$ on $\cn P^n$ satisfy
	$$\tau(\phi_{jk})=-\,4(n+1)\cdot\phi_{jk}\ \ \text{and}\ \ \kappa(\phi_{jk},\phi_{lm})=-\,4\cdot \phi_{jk}\cdot\phi_{lm}.$$
\end{example}

With the following result from \cite{Gha-Gud-4}, the authors show that a given eigenfamily $\E$ can be used to produce a large collection $\P_d(\E)$ of such objects.

\begin{theorem}\label{theorem-polynomials}
	Let $(M,g)$ be a Riemannian manifold and the set of complex-valued functions  $$\E=\{\phi_i:M\to\cn\,|\,i=1,2,\dots,n\}$$ 
	be a finite eigenfamily i.e. there exist complex numbers $\lambda,\mu\in\cn$ such that for all $\phi,\psi\in\E$ $$\tau(\phi)=\lambda\cdot\phi\ \ \text{and}\ \ \kappa(\phi,\psi)=\mu\cdot\phi\,\psi.$$  
	Then the set of complex homogeneous polynomials of degree $d$
	$$\P_d(\E)=\{P:M\to\cn\,|\, P\in\cn[\phi_1,\phi_2,\dots,\phi_n],\, P(\alpha\cdot\phi)=\alpha^d\cdot P(\phi),\, \alpha\in\cn\}$$ 
	is an eigenfamily on $M$ such that for all $P,Q\in\P_d(\E)$ we have
	$$\tau(P)=(d\,\lambda+d(d-1)\,\mu)\cdot P\ \ \text{and}\ \ \kappa(P,Q)=d^2\mu\cdot P\, Q.$$
\end{theorem}

\section{New Eigenfamilies via Rational Exponents}
\label{section-rational}

For a positive rational number $d\in\qn^+$ let $\L^{n-1}_d$ be the lattice in $\qn^n$ with 
$$
\L^{n-1}_d=\{\alpha=(\alpha_1,\alpha_2,\dots,\alpha_n)\in\qn^n\,|\,\alpha_1+\alpha_2+\cdots +\alpha_n=d,\ \alpha_k\in \qn^+_0\}.
$$

\begin{lemma}\label{lemma-kappa}
Let $(M,g)$ be a Riemannian manifold and $\E$ be a $(\lambda,\mu)$-eigenfamily on $M$.  If $\alpha,\beta\in\qn^+$ are two positive rational numbers and $\phi,\psi\in\E$, then the conformality operator $\kappa$ satisfies
$$
\kappa(\phi^\alpha,\psi^\beta)
=\mu\,\alpha\beta\cdot\phi^\alpha\psi^\beta.
$$
\end{lemma}

\begin{proof}
Let $\{X_1,X_2,\dots,X_m\}$ be a local orthonormal frame for the tangent bundle $TM$ of $M$.  Then
\begin{eqnarray*}
\kappa(\phi^\alpha,\psi^\beta)
&=&\sum_{j=1}^mX_j(\phi^\alpha)\cdot X_j(\psi^\beta)\\
&=&\sum_{j=1}^m\alpha\,\phi^{\alpha-1}X_j(\phi)\cdot\beta\,\psi^{\beta-1}X_j(\psi)\\
&=&\alpha\beta\,\phi^{\alpha-1}\psi^{\beta-1}\,\kappa(\phi,\psi)\\
&=&\mu\,\alpha\beta\,\phi^\alpha\psi^\beta.
\end{eqnarray*}
\end{proof}

\begin{lemma}\label{lemma-tau}
Let $(M,g)$ be a Riemannian manifold and $\E$ be a $(\lambda,\mu)$-eigenfamily on $M$.  If $\alpha,\beta\in\qn^+$ are two positive rational numbers and $\phi,\psi\in\E$, then the tension field $\tau$ satisfies
$$
\tau(\phi^\alpha)=(\lambda\cdot\alpha +\mu\cdot\alpha(\alpha-1))\cdot\phi^\alpha.
$$
$$
\tau(\phi^\alpha\,\psi^\beta)=
\lambda\cdot(\alpha+\beta)
+\mu\cdot (\alpha+\beta)(\alpha+\beta-1)
\cdot\phi^\alpha\,\psi^\beta,
$$
\end{lemma}

\begin{proof}
Let $x\in M$ be an arbitrary element and $\{X_1,X_2,\dots,X_m\}$ be a local orthonormal frame for the tangent bundle $TM$ of $M$ in a neighbourhood of $x$ such that $(\nab{X_k}{X_k})(x)=0$.  At the point $x\in M$, we then have 
\begin{eqnarray*}
\tau(\phi^\alpha)
&=&\sum_{j=1}^mX_j^2(\phi^\alpha)-(\nab{X_k}{X_k})(\phi^\alpha)\\
&=&\sum_{j=1}^mX_j(\alpha X_j(\phi)\cdot\phi^{\alpha-1})\\
&=&\alpha\sum_{j=1}^mX^2_j(\phi)\cdot\phi^{\alpha-1}+X_j(\phi)\cdot X_j(\phi^{\alpha-1}))\\
&=&\alpha(\tau(\phi)\cdot\phi^{\alpha-1}+(\alpha-1)\sum_{j=1}^mX_j(\phi)\cdot X_j(\phi)\cdot \phi^{\alpha-2}))\\
&=&(\lambda\,\alpha+\mu\cdot\alpha(\alpha-1))\cdot \phi^\alpha.
\end{eqnarray*}
\end{proof}

We are now ready to prove our main result presented in Theorem \ref{theorem-rational}.

\begin{proof}
Here we shall make repeated applications of Equations (\ref{equation-basic-1}), (\ref{equation-basic-2}), Lemma \ref{lemma-kappa} and Lemma \ref{lemma-tau}.  For the conformality operator $\kappa$ we yield
\begin{eqnarray*}
\kappa(\Phi_\alpha,\Phi_\beta)
&=&\kappa(\prod_{j=1}^n\phi_j^{\alpha_k},\prod_{k=1}^n\phi_k^{\beta_k})\\
&=&\sum_{j,k=1}^n\prod_{r\neq j}\phi_r^{\alpha_r}\prod_{s\neq k}\phi_s^{\beta_s}
\cdot\kappa(\phi_j^{\alpha_j},\phi_k^{\beta_k})\\
&=&\sum_{j,k=1}^n\prod_{r\neq j}\phi_r^{\alpha_r}\prod_{s\neq k}\phi_s^{\beta_s}
\cdot{\alpha_j}{\beta_k}
\cdot\phi_j^{\alpha_j-1}\phi_k^{\beta_k-1}\cdot\kappa(\phi_j,\phi_k)\\
&=&\mu\cdot \Phi_\alpha\,\Phi_\beta\sum_{j,k=1}^n\alpha_j\beta_k\\
&=&\mu\,d^2\cdot\Phi_\alpha\,\Phi_\beta.
\end{eqnarray*}
Employing the induction hypothesis we then get the following for the tension field $\tau$
\begin{eqnarray*}
\tau(\Phi_\alpha)
&=&\tau(\phi^{\alpha_{n}}_{n}\cdot\prod_{j=1}^{n-1}\phi_j^{\alpha_j})\\
&=&\tau(\phi^{\alpha_{n}}_{n})\cdot\prod_{j=1}^{n-1}\phi_j^{\alpha_j}
+2\cdot\kappa(\phi^{\alpha_{n}}_{n},\prod_{j=1}^{n-1}\phi_j^{\alpha_j})+
\phi^{\alpha_{n}}_{n}\cdot\tau(\prod_{j=1}^{n-1}\phi_j^{\alpha_j})\\
&=&(\lambda\cdot\alpha_{n}+\mu\cdot\alpha_n(\alpha_n-1))\cdot\phi_n^{\alpha_n}\cdot\prod_{j=1}^{n-1}\phi_j^{\alpha_j}\\
&&\qquad +\,2\cdot\kappa(\phi_n^{\alpha_n},\prod_{j=1}^{n-1}\phi_j^{\alpha_j})+\phi_n^{\alpha_n}\cdot\tau(\prod_{j=1}^{n-1}\phi_j^{\alpha_j})\\
&=&(\lambda\cdot\alpha_{n}+\mu\cdot\alpha_n(\alpha_n-1))\cdot\Phi_\alpha\\
&&\qquad  +\,2\,\alpha_n\,(\alpha_1+\cdots +\alpha_{n-1})\cdot\Phi_\alpha\\
&&+(\alpha_1+\cdots+\alpha_{n-1})(\lambda+\mu(\alpha_1+\dots +\alpha_{n-1}-1))\cdot\Phi_\alpha\\
&=&(\lambda\cdot d+\mu\cdot d(d-1))\Phi_\alpha.
\end{eqnarray*}
\end{proof}

\section{New Harmonic Morphisms}
\label{section-examples}

In this section we show how our observations yield new harmonic morphisms defined on the standard odd-dimensional spheres $S^{2m-1}$ in the Euclidean space $\rn^{2m}$.

\begin{example}
Let $S^{2m-1}$ be the odd-dimensional unit sphere in the standard Euclidean space $\rn^{2m}\cong\cn^{m}=\{(z_1,z_2,\dots,z_m)\,|\, z_j\in\cn\}$.  According to Example \ref{example-basic-sphere} the restrictions $z_1,z_2,\dots,z_m:S^{2m-1}\to\cn$ of the coordinate functions to the unit sphere form a $(\lambda,\mu)$-eigenfamily on $S^{2m-1}$.  

For a positive rational number $d\in\qn^+$ let $\L^{n-1}_d$ be the lattice in $\qn^n$ with 
$$
\L^{n-1}_d=\{\alpha=(\alpha_1,\alpha_2,\dots,\alpha_n)\in\qn^n\,|\,\alpha_1+\alpha_2+\cdots +\alpha_n=d,\ \alpha_k\in\qn^+_0\}.
$$
Then 
$$
\E_d^{n-1}=\{\Phi_\alpha=z_1^{\alpha_1}z_2^{\alpha_2}\cdots z_n^{\alpha_n}:S^{2m-1}\to\cn\,|\,\alpha\in\L_d^{n-1}\}
$$
is an eigenfamily on the odd dimensional sphere $S^{2m-1}$. If $\{\Phi_1,\Phi_2,\dots,\Phi_t\}$ is a finite subset of $\E_d^{n-1}$ and $P,Q:\cn^t\to\cn$ are linearily independent homogeneous polynomials of the same positive degree then the quotient
$$
\frac{P(\Phi_1,\dots ,\Phi_t)}{Q(\Phi_1,\dots ,\Phi_t)}
$$ 
is a non-constant harmonic morphism on the open and dense subset 
$$
\Omega (Q)=\{x\in M\,|\, Q(\Phi_1(p),\dots ,\Phi_t(p))\neq 0\}.
$$
\end{example}

\end{document}